 \documentclass[10pt, reqno]{amsart}
 \usepackage{tikz}
 \usepackage{hyperref, mathtools}
 \usepackage[skip=4pt]{caption}
 \usepackage{amssymb,amsmath,graphicx,amsthm}
  \usetikzlibrary{positioning}
  \usepackage{tkz-euclide}
\usepackage{rotating}
\usepackage{color,xcolor}
\definecolor{darkred}{rgb}{1,0,0} 
\definecolor{darkgreen}{rgb}{0,0.8,0}
\definecolor{darkblue}{rgb}{0,0,1}

\hypersetup{colorlinks,
linkcolor=darkblue,
filecolor=darkgreen,
urlcolor=darkred,
citecolor=darkgreen}

 \usetikzlibrary{calc,intersections,through,backgrounds}

 \def\bt{\begin{theorem}}
 	\def\el{\end{lemma}}
 \def\bl{\begin{lemma}}
 	\def\et{\end{theorem}}
 \def\bp{\begin{proposition}}
 	\def\ep{\end{proposition}}
 \def\bd{\begin{definition}}
 	\def\ed{\end{definition}}
 \def\br{\begin{remark}}
 	\def\er{\end{remark}}

 \def\R{{\mathbb R}}

 \def\C{\mathbb C}
 \def\B{\mathbb B}

 \def\P{{\mathbb P}}

 \def\label#1{\label{#1}}

 \numberwithin{equation}{section}

 \theoremstyle{plain}

 \newtheorem{theorem}{Theorem}[section]

 \newtheorem{lemma}[theorem]{Lemma}
 
 \newtheorem{proposition}[theorem]{Proposition}

 \theoremstyle{definition}
 
 \newtheorem{definition}[theorem]{Definition}
 
 \theoremstyle{remark}

 \newtheorem{remark}[theorem]{Remark}

 \newtheoremstyle{named}{}{}{\itshape}{}{\bfseries}{.}{.5em}{\thmnote{#3 }#1}
\theoremstyle{named}

 \DeclareMathOperator{\Rre}{Re}

 \newcommand{\p}{\partial}

 \newcommand\restrict[1]{\raisebox{-.5ex}{$|$}_{#1}}

 \def\Re{\text{Re}}

\begin{document}
 	
 	\title[Orthogonal testing families and holomorphic extension]{Orthogonal testing families and holomorphic extension from the sphere to the ball }      
\author{Luca Baracco}
\address{Dipartimento di Matematica Tullio Levi-Civita, Universit\`a di Padova, via Trieste 63, 35121 Padova, Italy}
\email{baracco@math.unipd.it}

\author{Martino Fassina}
\address{Department of Mathematics, University of Illinois, 1409 West Green
Street, Urbana, IL 61801, USA}
\email{fassina2@illinois.edu}
\thanks{The second author acknowledges support of NSF grant 13-61001.}

 \begin{abstract} 

Let $\mathbb{B}^2$ denote the open unit ball in $\mathbb{C}^2$, and let $p\in \mathbb{C}^2$\textbackslash$\overline{\mathbb{B}^2}$. We prove that if $f$ is an analytic function on the sphere $\partial\mathbb{B}^2$ that extends 
holomorphically in each variable separately and along each complex line through $p$, then $f$ is the trace of a holomorphic function in the ball.  

\end{abstract}
  	\maketitle
	
 	\section{Introduction and Main Theorem}
 	\label{s1}
It is a well-known fact in the theory of several complex variables that a function is holomorphic if and only if it is holomorphic in each variable separately. This result goes back to Hartogs \cite{H}. It is natural to consider a boundary version of Hartogs' theorem. The general problem is to take
a boundary function and ask if holomorphic extensions on vertical and horizontal
slices are enough to guarantee an extension which is holomorphic in both variables simultaneously. In \cite{Lw} Lawrence proved that vertical and horizontal slices are enough to detect the existence of holomorphic extension to the interior for functions defined on a small perturbation of the boundary of the unit ball $\B^2\subset\C^2$. However, the result is not true for the ball itself, for which additional conditions are needed. 

There is a vast literature on describing families of directions which suffice for testing analytic extension of a continuous function $f$ from the sphere to the ball. The first significant result was obtained by Stout \cite{S}, who
used as testing family all the straight lines. Reducing the testing family, Agranovsky and
Semenov \cite{AS} used the lines which meet an open subset of the ball, Rudin \cite{R} the lines tangent to a concentric subsphere, Baracco, Tumanov and Zampieri \cite{BTZ} the lines tangent to any strictly convex subset of $\B^2$. Among the many contributions to the problem we mention \cite{AV71, B12, D99, G12, Lw18} and \cite{T07}.

It is well known that the lines which meet a single point do not suffice. With additional hypotheses on the initial regularity of $f$ on $\p\B^2$ (namely, for $f$ analytic rather than just continuous) one can prove that the families of lines through the following sets of points do suffice: two interior points (Agranovsky \cite{A11}), one boundary point (Baracco \cite{B16}), two points in $\C^2\setminus\overline{\B^2}$ whose joining line is tangent to the sphere (Baracco and Pinton \cite{BP}). 

In \cite{B13} Baracco proved, for $f$ continuous on $\p\B^2$, that three non-aligned points in the ball suffice. The result was later improved by Globevnik \cite{Gl12}, who allowed the points to lie outside $\B^2$ provided that at least one of the joining lines meets the ball. At the end of his paper, Globevnik asked the following question: let $a,b,c\in\C^2$ be three points whose joining lines do not meet the ball. Do the lines through $a$, $b$ and $c$ constitute a testing family for holomorphic extension?

In this paper, we give a partial answer to Globevnik's question, under the assumption that $f$ is analytic on $\p\B^2$. Here is our main result.
\begin{theorem}\label{maintheorem}
 Let $f$ be an analytic function on the sphere $\partial\mathbb{B}^2$ which extends holomorphically in each variable separately and along each complex line through a point 
$p\in\mathbb{C}^2\setminus\overline{\mathbb{B}^2}$. Then $f$ extends holomorphically to $\mathbb{B}^2$.
\end{theorem}

We will only prove the case $p=(p_1,p_2), |p_1|>1, |p_2|>1$, which falls into Globevnik's question, with two of the points being at infinity. We will not deal with the cases $|p_1|<1$ or $|p_2|<1$, which were already treated by Globevnik \cite{Gl12}.

For the proof of our theorem we employ techniques related to stationary discs in the sense of Lempert \cite{L} that have already been used in this context in \cite{BTZ,B13,B16}. We add a better understanding of the geometry of the space of lifts of stationary discs and the use of a continuity principle.

\section{Stationary Discs}
In this section we summarize some basic facts on stationary discs and we prove a technical lemma that will be used in the proof of Theorem \ref{maintheorem}. For more background information on analytic and stationary discs, we refer the reader to the original paper of Lempert \cite{L} and Tumanov's lecture notes \cite{T}.

Let $M$ be a smooth real manifold in $\C^n$ and let $T M$ denote its tangent bundle. For $p\in M$ recall the space $T_p^{1,0}M\subset T_p M\otimes\C$ of complex $(1,0)$-vectors defined as $$T_p^{1,0}M:=\big{\{}X\in T_pM\otimes\C\colon X=\sum a_j\,\partial/\partial z_j\big{\}}.$$

Let $T^*\C^n$ be the real cotangent bundle of $\C^n$. Since every $(1,0)$-form is uniquely determined by its real part, we represent $T^*\C^n$ as the space of $(1,0)$-forms on $\C^n$. More precisely, for $z\in\C^n$, we use the identification 
\begin{equation*}
T_z^*\C^n\simeq (T_z^{1,0}\C^n)^*
\end{equation*}
\begin{equation*}
\omega\xrightarrow{\sim} \Omega
\end{equation*}
where $\langle \omega, X\rangle=\Rre\,\langle \Omega,X\rangle$ for all $X\in T_z\C^n$. Let $T^*_{M}\C^n\subset T^*\C^n$ be the real {\em conormal bundle} of $M$. Using the representation of $T^*\C^n$ by $(1,0)$-forms, we define the fiber $(T^*_{M}\C^n)_p$ at $p\in M$ as 
\begin{eqnarray*} 
\left(T^*_{M}\C^n\right)_p:=\{ \omega\in T_p^*\C^n \colon \Re\, \omega\restrict {T_p M} =0 \}.
\end{eqnarray*}  
Note that if $r$ is a defining function for $M$ then the conormal bundle $T^*_{M}\C^n$ is generated by $\partial r=\sum \partial r/\partial z_j dz_j$.

Let $\Delta$ be the unit disc in $\mathbb{C}$. An {\em analytic disc} in a complex manifold $X$ is a holomorphic map $A:\Delta\rightarrow X$. 
We say that $A$ is {\em attached} to some set $M\subset X$ if $A$ is continuous in the closed disc $\overline{\Delta}$ and $A(\partial\Delta)\subset M$.

Let $D$ be a strictly pseudoconvex domain in $\C^n$. An analytic disc $A$ attached to $\p D$ is said to be {\em stationary} if there exists a map $\lambda\colon\partial\Delta\rightarrow\R_{> 0}$ such that the function $\tau\lambda(\tau)\partial r(A(\tau))$, defined for $\tau\in\partial\Delta$, extends to a function continuous in $\overline{\Delta}$ and holomorphic in $\Delta$. In other words, a disc is stationary if it admits a meromorphic ``lift" to a disc in the cotangent bundle attached to the conormal bundle.

Let $\B^n$ be the open unit ball in $\C^n$. It is immediate to verify that the conormal bundle of the $n$-sphere $\p\B^n$ is given by 
\begin{equation*}
T^*_{\p\B^n}\C^n=\{(z,\lambda \overline{z}), z\in\p\B^n,\, \lambda\in\R\}.
\end{equation*}
In this case, the stationary discs are precisely the straight ones, that is, the ones obtained by intersecting the ball with complex lines.

The following two propositions are well known.
\begin{proposition}
Let $A\colon \Delta\rightarrow X$ be a stationary disc and let $\varphi\colon\Delta\rightarrow\Delta$ be an automorphism of the unit disc. Then $A\circ\varphi\colon\Delta\rightarrow X$ is also stationary.
\end{proposition}
\begin{proof}
Let $\varphi$ be given by $$\varphi(\tau)=\alpha\,\frac{\tau-a}{1-\tau\overline{a}}$$
for some $a,\alpha\in\C$ with $|\alpha|=1,|a|<1.$ The proposition is proved if we can find a map $\widetilde{\lambda}\colon\partial\Delta\rightarrow\R_{> 0}$ such that the function $\tau\widetilde{\lambda}(\tau)\partial r(A(\varphi(\tau)))$, defined for $\tau\in\partial\Delta$, extends to a function continuous in $\overline{\Delta}$ and holomorphic in $\Delta$. For $\tau\in \p\Delta$, let $$\widetilde{\lambda}(\tau):=|\tau-a|^2\lambda(\varphi(\tau)).$$ On $\p\Delta$ we have $$\tau\widetilde{\lambda}(\tau)\partial r(A(\varphi(\tau)))=(1-\tau\overline{a})^2\underbrace{\frac{(\tau-a)}{(1-\tau\overline{a})}\lambda(\varphi(\tau))A(\varphi(\tau))}_{H(\tau)}.$$
Since $A$ is stationary, $H(\tau)$ extends to a function in $\overline{\Delta}$ holomorphic in $\Delta$, and the same is true for $(1-\tau\overline{a})^2H(\tau)$.
\end{proof}

\begin{proposition}
Let $\lambda_1,\lambda_2\colon\partial\Delta\rightarrow\R_{> 0}$ be such that both $\tau\lambda_1(\tau)\partial r(A(\tau))$ and $\tau\lambda_2(\tau)\partial r(A(\tau))$ extend to functions continuous in $\overline{\Delta}$ and holomorphic in $\Delta$. Assume also $\lambda_1(1)=\lambda_2(1)$. Then $\lambda_1=\lambda_2$.
\end{proposition}
\begin{proof}
Since $\partial r$ generates the conormal bundle, we have $$\Rre\,\langle\partial r(A),\partial_{\theta}A(e^{i\theta})\rangle=0.$$
Therefore
\begin{equation}\label{zero}
\Rre\,\langle(\lambda_1-\lambda_2)\partial r(A),ie^{i\theta}A'(e^{i\theta})\rangle=0.
\end{equation}
Equation \eqref{zero} implies that the holomorphic function $\langle(\lambda_1-\lambda_2)\partial r(A),ie^{i\theta}A'(e^{i\theta})\rangle$ is constant, and therefore identically zero (since it vanishes at 1). Hence
$$(\lambda_1-\lambda_2)\langle \partial r(A),i\tau A'(\tau)\rangle\equiv 0.$$
By strong pseudoconvexity of $D$ and the Hopf Lemma, $\langle \partial r(A),i\tau A'(\tau)\rangle$ is nonvanishing on $\p\Delta$. Hence $\lambda_1=\lambda_2$.
\end{proof}

The discussion above shows that the lift of a stationary disc is unique up to multiplication by a scalar function. It is therefore natural to think of a lift as a geometric object in the projective space $\P T^*\C^n$. In the rest of the paper we will use $[\,,\,]$ to denote projective coordinates.

We now restrict our attention to $\C^2$. In particular, for the proof of Theorem \ref{maintheorem} we will need an explicit formula for the union of the lifts of the discs obtained by intersecting the unit ball $\mathbb{B}^2$ with the complex lines through a point. Let us fix some notation. For a given point $p\in \mathbb{C}^2\setminus\overline{\mathbb{B}^2}$, we consider the family of complex lines through $p$. For each $z\in\mathbb{B}^2$ let $A_z$ be the disc obtained by intersecting $\mathbb{B}^2$ with the line through $p$ and $z$. Each such disc $A_z$ is stationary. We denote by $A^*_z$ the corresponding lift in the (projectivized) cotangent bundle.
The next lemma gives a precise description of the set of all lifts $M_p:=\bigcup_{z\in\mathbb{B}^2} A^*_{z}$. 

\begin{lemma}\label{prop}
 For $p\in \mathbb{C}^2\setminus\overline{\mathbb{B}^2}$ the following holds:
\begin{equation*}\label{rialzata} 
M_p=\{ (z;[\overline{z}(z\cdot\overline{p} -1)+\overline{p} (1-|z|^2)])\in \B^2\times \P T^*\C^2  \}.
\end{equation*}
\end{lemma}
\begin{proof}
 For $z\in \B^2$, the complex line through $z$ and $p$ consists of points of the form $z+\zeta(p-z),\,\zeta\in\C$. By intersecting with the unit ball, we find that the points of the disc $A_z$ are the ones for which $\zeta$ satisfies
\begin{equation*}\label{disco1}|\zeta|^2+2\Re \left( \zeta\frac{(p-z)\cdot \overline{z}}{|p-z|^2} \right) +\frac{|z|^2 -1}{|p-z|^2}\le 0 .\end{equation*}
Let $R_z:= \sqrt{\frac{|p|^2+|z|^2 +|p\cdot \overline{z} |^2 -|z|^2|p|^2 -2\Re ( p\cdot\overline{z})}{|z-p|^4} }$ and $C_z:= -\frac{z\cdot\overline{(p-z)}}{|p-z|^2}$. We can then parametrize $A_z$ 
over the unit disc by 
\begin{equation} \label{disco2}
 A_z(\tau)= z+(R_z\tau +C_z)(p-z),\quad \tau\in\Delta.
\end{equation}
Note that $R_z$ and $C_z$ satisfy the following relations:
\begin{equation}\label{rel1}
R_z^2=-\frac{|z|^2-1}{|p-z|^2}+\frac{\big{|}p\cdot\overline{z}-|z|^2\big{|}^2}{|p-z|^4}
\end{equation}
and
\begin{equation}\label{rel2}
-R^2_z+|C_z|^2=\frac{|z|^2-1}{|p-z|^2}\,.
\end{equation}
Moreover $A_z(-\frac{C_z}{R_z})=z$. From \eqref{disco2} we can see that the meromorphic lift attached to the conormal bundle is given by
$$ A^*_z (\tau)=\Big{(}A_z(\tau),\frac{\overline{z} \tau+(R_z+\overline{C_z} \tau)(\overline{p-z})}{\tau}\Big{)}.$$ 
Using \eqref{rel1} and \eqref{rel2} we obtain the following formula in the projectivized cotangent bundle: $$A^*_z\Big{(}-\frac{C_z}{R_z}\Big{)}=(z;[\overline{z}(z\cdot\overline{p} -1)+\overline{p} (1-|z|^2)]).$$
This concludes the proof.
\end{proof}

\begin{remark}\label{rmk}
 $M_p$ is a manifold of dimension $4$ foliated by complex curves. With a standard computation one can also see that $M_p$ is a CR manifold of dimension 1 at all points except for those that project over the complex line $\{z\in\C^2\colon
z\cdot\overline{p}=1\}$. We thus have a decomposition
$$ M_p= M_p^{reg} \cup M_p^{sing}, \text{ where }\,\, M_p^{sing} =\{ (z,[\overline{p}])\, \colon\, z\cdot\overline{p}=1 \}.$$
Note that the set of CR singular points of $M_p$ is a complex curve which intersects transversally each $A^*_z$.
\end{remark}

\begin{remark}\label{remarko}
A decomposition analogous to the one described in Remark \ref{rmk} holds even when $p$ is a point at infinity, that is, when we are considering all the lines parallel to a given direction. 
As an example, let us describe the case of the lines parallel to the $z_2$-direction. We denote by $(0,1)_\infty$ the corresponding point at infinity. In this situation, we can use $z_1 \in \Delta$ as a parameter:
for each $z_1$ we have the disc $A_{(z_1,0)}(\tau)=(z_1,\sqrt{1-|z_1|^2}\tau)$, which lifts to 
\begin{equation*}
\begin{split}
A^*_{(z_1,0)}((\tau)&=\big{(}(z_1,\sqrt{1-|z_1|^2}\tau);[\overline{z_1}\tau,\sqrt{1-|z_1|^2}]\big{)}\\
&=\big{(}(z_1,\sqrt{1-|z_1|^2}\tau);[\overline{z_1}\tau\sqrt{1-|z_1|^2},1-|z_1|^2]\big{)}.
\end{split}
\end{equation*}
We conclude that
 \begin{equation*}\label{parallele}
M_{(0,1)_\infty}=\{ \big{(}(z_1,z_2);[\overline{z_1}z_2,1-|z_1|^2]\big{)}\in \B^2\times \P T^*\C^2\}. 
\end{equation*} 
Note that $M_{(0,1)_\infty}$ is a CR manifold of CR dimension 1 at all the points for which $z_1\neq 0$. As before, we have a decomposition
$$ M_{(0,1)_\infty}= M_{(0,1)_\infty}^{reg} \cup M_{(0,1)_\infty}^{sing}, \text{ where \,} M_{(0,1)_\infty}^{sing}=\{\big{(}(z_1,0);[0,1]\big{)}\,\colon  z_1\in \Delta \}.$$ 
Similar formulas hold for the point at infinity $(1,0)_\infty$, which corresponds to the lines parallel to the $z_1$-direction. In particular, one can check that
\begin{equation*}
M_{(1,0)_\infty} =\{ \big{(}(z_1,z_2);[1-|z_2|^2, z_1 \overline{z_2}] \big{)}\in \B^2\times \P T^*\C^2\}. 
\end{equation*}
\end{remark}

\section{Proof of Theorem \ref{maintheorem}}

\begin{proof}

Let $E:=\mathbb{P}T^*_{\p\B^2}\C^2$. The function $f$ lifts naturally to a real analytic function $F:E\rightarrow \C$. Note that $E$ is maximally totally real in $\C^2\times \P^1_\C$, hence $F$ extends holomorphically to a neighborhood of $E$. From now on, we denote by $\zeta$ the projective coordinate.
By the hypotheses on the holomorphic extendibility of $f$, it is possible to lift $f$ to a function defined on the manifolds 
$M_{(1,0)_\infty}$, $M_{(0,1)_\infty}$ and $M_p$. Note in fact that the lifts $A_z^*$ in these three families do not intersect outside of $E$. Our goal is now to show that $F$ extends to a holomorphic function in $\B^2\times \P^1_{\C}$, and therefore $F$ is constant in $\zeta$. We can then conclude that $F$ projects down to a holomorphic function in $\B^2$ that extends $f$ holomorphically to the ball. 

We start by observing that the function $F$ is CR on the regular part of the three manifolds $M_{(1,0)_\infty}$, $M_{(0,1)_\infty}$ and $M_p$. Since $M_p^{reg}$ is foliated by complex curves and each curve has some points where $F$ extends holomorphically to a full neighborhood, we can apply the propagation theorem of Hanges and Treves \cite{HT} to conclude that $F$ extends holomorphically to a neighborhood of $M_p^{reg}$. We will focus on the disc $A_0^*$ in $M_p$, that is, the lift of the disc through $p$ and the origin. Figure \ref{fig1} illustrates the situation. There, the circle represents the projectivized conormal bundle $E$, and the diameter is $A_0^*$. The shaded regions correspond to the neighborhoods where we know $F$ extends holomorphically.

\begin{figure}[h]
\begin{turn}{-45}
\centering
 \begin{tikzpicture}[scale=0.8]
                        
\draw[black,thick] circle (3cm);
\draw[thin,dashed] circle (2.6cm);
\draw[thin,dashed] circle (3.4cm);
\draw [line width=0.7mm, black ] (0,-3) -- (0,3);

\coordinate  (M) at (30:1.5);
\coordinate (P) at (150:1.5);
\coordinate (L) at (30:2);
\coordinate (Q) at (150:2);
\draw[thick,dashed] [shorten >=-0.5cm,shorten <=-0.5cm]    (M) -- (P);
\draw[thick]    (M) -- (P);

\node[rotate=45] at (2.7,-1.7) {$E$};
\node[rotate=45] at (0.25,-0.4) {{\bf $A_0^*$}};
\node[rotate=45] at (-1.0,0.2) {{\bf $M_p^{sing}$}};

\path [draw=none,fill=gray, fill opacity = 0.1,even odd rule] (0,0) circle (3.4) (0,0) circle (2.6);

 \tkzDefPoint(30:1.5){R}  \tkzDefPoint(0,3){B}
  \tkzDefPoint(150:1.5){T}  \tkzDefPoint(0,-3){D}
  \tkzInterLL(D,B)(R,T) \tkzGetPoint{E}
\draw[fill] (E) circle [radius=0.1] ; 

\draw [thin,dashed, fill=gray, fill opacity = 0.1] (0,-3) to[out=55,in=-55] (E) to[out=235,in=-235] (0,-3);
\draw [thin,dashed, fill=gray, fill opacity = 0.1] (E) to[out=60,in=-60] (0,3) to[out=240,in=-240] (E);
\end{tikzpicture}
\end{turn}
\caption{} \label{fig1}
\end{figure}

The next step is to achieve holomorphic extension at the point of CR singularity (shown in Figure \ref{fig1}). To this end, we construct a continuous family $\{B_t\}$ of analytic discs in $\P T^*\C^2$  attached to $M_{(1,0)_\infty}^{reg}\cup\, \,M_{(0,1)_\infty}^{reg}$ such that: 
\begin{itemize}
\item for some value of $t$, the center of $B_t$ is at the point of CR singularity;
\item for some value of $t$, the disc $B_t$ is contained in the neighborhood of $E$ where $F$ is holomorphic. 
\end{itemize} 
Note that $F$ is holomorphic on the boundary of $B_t$ for all $t$, since the discs are attached to $M^{reg}_{(1,0)_{\infty}}\cup\,\, M^{reg}_{(0,1)_{\infty}}$. Hence, assuming that we have such a family of discs, we can apply the continuity principle to conclude that $F$ extends holomorphically in a full neighborhood of $A_0^*$. We now consider the (continuous) family of lines through $0$. Again by the continuity principle, $F$ extends holomorphically along the lift of each disc through $0$. Consequently, $F$ is holomorphic in a neighborhood of $\{0\}\times \P^1_{\C}$, which is what we wanted to prove.

The rest of the proof is entirely devoted to constructing a family of discs $\{B_t\}$ with the properties described above.
We start by recalling from Remark \ref{remarko} the formulas:
\begin{equation} \label{par1}
M_{(1,0)_\infty} =\{ \big{(}(z_1,z_2);[1-|z_2|^2, z_1 \overline{z_2}] \big{)}\in \B^2\times \P T^*\C^2\} 
\end{equation}
\begin{equation}\label{par21}
M_{(0,1)_\infty}=\{ \big{(}(z_1,z_2) ;[z_2\overline{z_1}, 1-|z_1|^2] \big{)} \in \B^2\times \P T^*\C^2 \}.
\end{equation}
The set $E$, which is the only part shared by these two manifolds, is given by 
\begin{equation*}
\big{(}re^{i\eta_1},\sqrt{1-r^2}e^{i\eta _2};[re^{-i\eta_1},\sqrt{1-r^2} e^{-i\eta_2}]\big{)} \,\,\, 0\le r \le 1,\, \,\,0\le \eta_1,\eta_2\le 2\pi .
\end{equation*}
Dividing out by the last term, we now introduce a complex coordinate on the projective component. For $r\neq 1$ we then have a parametrization of $E$ given by
\begin{equation}\label{par2}  \bigg{(}re^{i\eta_1},\sqrt{1-r^2}e^{i\eta _2},\frac{r}{\sqrt{1-r^2}}e^{i(\eta_2-\eta_1)}\bigg{)}\,\,\,\, \,0< r < 1,\, \,\,\,0\le \eta_1,\eta_2\le 2\pi .\end{equation}
Equations \eqref{par2} and \eqref{par1} imply that
\begin{equation}\label{primoman}
 \bigg{(}r\rho_1 e^{i\eta_1},\sqrt{1-r^2}e^{i\eta_2}, \frac{r^2 \sqrt{1-r^2}e^{i\eta_2}}{r\rho_1 e^{i\eta_1}(1-r^2)}\bigg{)} \,\,\,\,\, 0<r,\rho_1<1, \,\,\,\,\eta_1,\eta_2 \in \R
\end{equation}
is a parametrization for (almost) all $M^{reg}_{(1,0)_\infty}$. Analogously, from \eqref{par2} and \eqref{par21},
\begin{equation}\label{secondoman}
 \bigg{(}r e^{i\eta_1},\sqrt{1-r^2}\rho_2 e^{i\eta_2}, \frac{r^2 \sqrt{1-r^2}\rho_2 e^{i\eta_2}}{r e^{i\eta_1}(1-r^2)}\bigg{)} \,\,\,\,\, 0<r,\rho_2<1,\,\,\,\, \eta_1,\eta_2 \in \R
\end{equation}
is a parametrization for $M^{reg}_{(0,1)_\infty}$.
Equations \eqref{primoman} and \eqref{secondoman} together give the parametrization
\begin{equation} \label{globale}
\phi(r,\rho_1,\eta_1,\rho_2,\eta_2)= \bigg{(}r\rho_1 e^{i\eta_1},\sqrt{1-r^2}\rho_2e^{i\eta_2}, \frac{r^2 \sqrt{1-r^2}\rho_2e^{i\eta_2}}{r\rho_1 e^{i\eta_1}(1-r^2)}\bigg{)}
\end{equation}
\begin{equation*}
 0<r,\rho_1,\rho_2<1,\,\,\,\, \eta_1,\eta_2 \in \R.
 \end{equation*}
Note that  $\phi(r,1,\eta_1,\rho_2,\eta_2)\in M_{(0,1)_\infty}$ and $\phi(r,\rho_1,\eta_1,1,\eta_2)\in M_{(1,0)_\infty}$.

We now look for functions
$\rho_1,\rho_2,\eta_1,\eta_2 : \partial\Delta\rightarrow \R$ such that 
\begin{eqnarray}\label{condizionii}
\begin{cases}
\phi\big{(}r,\rho_1(e^{i\theta}),\rho_2(e^{i\theta}),\eta_1(e^{i\theta}),\eta_2(e^{i\theta})\big{)} \text{ extends holomorphically to $\Delta$,}\\
\text{ for each $\theta$, $\rho_1(e^{i\theta})=1$ or $\rho_2(e^{i\theta})=1$  } \\
\end{cases}
\end{eqnarray}
To satisfy the first condition of \eqref{condizionii}, the function $\rho_1(e^{i\theta})e^{i\eta_1(e^{i\theta})}$ has to extend holomorphically. This happens if and only if 
\begin{eqnarray} \label{bishop}
\eta_1(e^{i\theta})= T_1 \log(\rho_1(e^{i\theta})) + \psi_1.
\end{eqnarray}
Analogously, looking at the second component, we obtain the condition
\begin{eqnarray}\label{bishop2}
\eta_2(e^{i\theta})= T_1 \log(\rho_2(e^{i\theta})) + \psi_2.
\end{eqnarray}
Here $\psi_1$ and $\psi_2$ are constants and $T_1$ is the Hilbert transform on the unit disc normalized by the condition $T_1 u (1)=0$.
Note that the third component of $\phi$ is automatically holomorphic when the first two components are holomorphic and never zero. Let now $p=(p_1,p_2)$, and let $t\in \big{[}\frac{1}{|p|^2},\frac{1}{|p|}\big{)}$. Then $tp=(tp_1,tp_2) \in \B^2$. Lemma \ref{prop} implies that the point  $Q_t=\big{(}(tp_1,tp_2);[\overline{p} (t+1-2t^2|p|^2)]\big{)}$ is the only point in $M_p$ such that $\pi (Q_t)=tp$, where $\pi$ is the natural projection from the cotangent bundle. We now look for a family of analytic discs $\{B_t\}$ (see Figure \ref{fig2}) such that 
\begin{eqnarray}\label{condizioni}
\begin{cases}
B_t(\partial\Delta )\subset M_{(1,0)_\infty}^{reg}\cup\,\, M_{(0,1)_\infty}^{reg}  \\
  B_t(0)=Q_t \,\,\,\, \forall \ t.
  \end{cases}
\end{eqnarray}

\begin{figure}[h]
 		\centering
 		\begin{tikzpicture}[scale=0.6]
 		\draw (0,0) parabola (-1,6);
 		\draw (3,3) parabola (2,9);
 		\draw (-1,-1) -- (4,4);
 		\draw (0,0) parabola (6,-1);
 		\draw (3,3) parabola (9,2);
 		\draw [line width=0.7mm, black ](1,1)..controls (1.4,3.5) and (1.6,3.5) .. (2,2);
 		\draw [line width=0.7mm, black ](1,1)..controls (3.5,1.4) and (3.5,1.6) .. (2,2);
\node at (7,0) {$M^{reg}_{(1,0)_{\infty}}$};
\node at (0,7) {$M^{reg}_{(0,1)_{\infty}}$};
\node at (1.6,3.5) {$B_t$};
\node at (4.2,4) {$E$};
 		\end{tikzpicture}
 		\caption{}\label{fig2}
 \end{figure}
 
Let $\rho_j, \eta_j$ for $j=1,2$ be solutions of \eqref{bishop} and \eqref{bishop2}. We want to determine conditions on the $\rho_j$ and $\eta_j$ such that 
\eqref{condizioni} holds. Looking at the second equation in \eqref{condizioni}, we let $\alpha_j:=\widetilde{\rho_j e^{i\eta_j}}(0)$ for $j=1,2$, where $\,\widetilde{\cdot}\,$ denotes the holomorphic extension to the unit disc.
The $\alpha_j$ must satisfy:
\begin{eqnarray} \label{sistema}
\begin{cases}
r\alpha_1=tp_1 \\
\sqrt{1-r^2}\alpha_2=tp_2 \\
\frac{r\alpha_2}{\sqrt{1-r^2}\alpha_1}=\frac{\overline{p_1}}{\overline{p_2}}.
\end{cases}
\end{eqnarray}
The solution to \eqref{sistema} is given by
\begin{eqnarray*}\label{soluzioni}
\begin{cases}
r=\frac{|p_1|}{|p|} \\
\alpha_1=t|p| \frac{p_1}{|p_1|} \\
\alpha_2=t|p|\frac{p_2}{|p_2|}.
\end{cases}
\end{eqnarray*}
Let $\big{\{}\rho^t_j:\partial\Delta\rightarrow (0,1],\, t\in \big{[}\frac{1}{|p|^2},\frac{1}{|p|}\big{)}\big{\}}$ for $j=1,2$ be two continuous families of smooth functions such that, for all $t$, the following conditions are satisfied:
\begin{eqnarray*} \label{tcomponenti}
\begin{cases}
\rho^t_1(e^{i\theta})=1\text{ for } 0\le \theta\le \pi  \\
\rho^t_2(e^{i\theta})=1\text{ for }\pi \le \theta\le 2\pi \\
\frac{1}{2\pi}\int^{2\pi}_0 \log(\rho^t_j(e^{i\theta}))d\theta=\log(t|p|)\ j=1,2.
\end{cases}
\end{eqnarray*}
Moreover, for each $t$, let $\psi_j^t$ for $j=1,2$ be constants such that
\begin{eqnarray*} \label{offset}
\psi_j^t+\frac{1}{2\pi}\int^{2\pi}_0 T_1\log(\rho^t_j(e^{i\theta})) \,d\theta =\arg (p_j).
\end{eqnarray*}
Now choose two families of functions $\big{\{}\eta_j^t, \,t\in \big{[}\frac{1}{|p|^2},\frac{1}{|p|}\big{)}\big{\}}$ such that \eqref{bishop} and \eqref{bishop2} are both satisfied for each $t$, with $\rho_1, \rho_2, \psi_1,\psi_2$ replaced by $\rho^t_1, \rho^t_2, \psi^t_1,\psi^t_2$. We thus obtain the family of discs $B_t=\phi(r,\rho_1^t,\eta_1^t,\rho_2^t,\eta_2^t)$ satisfying \eqref{condizioni}.

Note that, for $t \to \frac{1}{|p|}$, the disc $B_t$ shrinks to the point $(p,\frac{\overline{p_1}}{\overline{p_2}})$. Therefore, for values of $t$ close to $\frac{1}{|p|}$, the disc $B_t$ is contained in the neighborhood of $E$ where $F$ is holomorphic.

\end{proof}

\begin{remark}
The hypothesis of orthogonality of the testing families was used in the construction of the discs $\{B_t\}$. With that assumption, it was possible to attach the discs to $M_{(1,0)_\infty}^{reg}\cup\,\, M_{(0,1)_\infty}^{reg}$ by elementary techniques.
\end{remark}

 								\bibliographystyle{alpha}

 							\end{document}